\documentclass[11pt,a4paper,reqno]{amsart}
\usepackage[T1]{fontenc}
\usepackage[utf8]{inputenc}
\usepackage{lmodern}
\usepackage{amsmath,amssymb,amsthm,mathrsfs}
\usepackage[left=30mm,right=30mm,top=28mm,bottom=29mm,headsep=20pt]{geometry}
\usepackage[expansion=false]{microtype}
\usepackage{etoolbox}
\usepackage[hidelinks]{hyperref}

\numberwithin{equation}{section}
\newtheorem{theorem}{Theorem}[section]
\newtheorem{lemma}[theorem]{Lemma}
\newtheorem{proposition}[theorem]{Proposition}
\newtheorem{corollary}[theorem]{Corollary}
\theoremstyle{remark}
\newtheorem{remark}[theorem]{Remark}

\DeclareMathOperator{\tr}{tr}
\DeclareMathOperator{\Scal}{Scal}
\DeclareMathOperator{\Vol}{Vol}

\title[Second pinching and a moment-defect gap]
{An improved second pinching theorem\\
for minimal hypersurfaces with\\
constant scalar curvature in spheres}
\author[Ruihan Chen]{Ruihan Chen\\
\normalfont Faculty of Science, The University of Hong Kong}
\makeatletter
\patchcmd{\@setauthors}{\MakeUppercase{\authors}}{\authors}{}{}
\def\@evenhead{\normalfont\scriptsize
  \rlap{\thepage}\hfil\shortauthors\hfil}
\makeatother
\date{September 5, 2026}
\hypersetup{
  pdftitle={An improved second pinching theorem for minimal hypersurfaces with constant scalar curvature in spheres},
  pdfauthor={Ruihan Chen},
  pdfsubject={Second pinching and an average quadratic moment-defect gap},
  pdfkeywords={Minimal hypersurface, constant scalar curvature, Chern conjecture, second pinching}
}

\begin{document}

\begin{abstract}
Let $M^n$ be a closed minimal hypersurface of the unit sphere with constant
squared norm $S$ of the second fundamental form. For $n\geq4$, we prove that
$S>n$ implies $S>85n/59$, improving the bound $10n/7$ of Suh and Yang.
No constancy assumption is imposed on the cubic trace $f_3$.
The proof combines a symmetrized Hessian estimate, an integral identity for
a mixed moment, and a weighted estimate for the repeated-index components
of the covariant derivative of the second fundamental form. An exact
polynomial argument treats the entire excluded interval. We also prove an
explicit positive lower bound for the average of the normalized quadratic
defect $P/S^2$, where $P=f_4-f_3^2/S-S^2/n$. This second estimate holds in
every dimension $n\geq2$ and gives a quantitative obstruction to
concentration near the locus of two principal curvatures.
\end{abstract}

\maketitle

\section{Introduction and main results}
\label{sec:1}

Let $F:M^n\to\mathbb{S}^{n+1}(1)$ be a smooth minimal immersion of a closed
connected manifold. Here and throughout, closed means compact without
boundary. Denote the shape operator by $A$, the second fundamental form
by $h$, and the principal curvatures by $\lambda_1,\ldots,\lambda_n$. Set
\begin{equation}
 S=|h|^2=\tr(A^2),\qquad f_r=\tr(A^r).
 \label{eq:1.1}
\end{equation}
The Gauss equation gives $\Scal_M=n(n-1)-S$, so constant scalar curvature
is equivalent to constant $S$.

Simons' identity \cite{ref11} implies that a constant value of $S$ in $[0,n]$
is either $0$ or $n$. In the first case the immersion is totally geodesic.
In the second case its image is, up to an ambient isometry, a Clifford
minimal hypersurface
\begin{equation}
 \mathbb{S}^m\!\left(\sqrt{\frac{m}{n}}\right)
 \times\mathbb{S}^{n-m}\!\left(\sqrt{\frac{n-m}{n}}\right),
 \qquad 1\leq m\leq n-1,
 \label{eq:1.2}
\end{equation}
by the classification of Chern--do Carmo--Kobayashi and Lawson
\cite{ref1,ref8}. For an immersion, this statement concerns its image;
the source may cover that image.

Chern's discreteness conjecture asks whether, for fixed $n$, the constant
values of $S$ arising in this way form a discrete set. The second pinching
problem asks whether
\begin{equation}
 S>n\quad\Longrightarrow\quad S\geq2n.
 \label{eq:1.3}
\end{equation}
Peng and Terng introduced the higher-order identities used to study this
problem \cite{ref9,ref10}. Yang and Cheng obtained successive improvements
of the second gap \cite{ref14,ref15}. Suh and Yang \cite{ref12} proved the
estimate $S>10n/7$ whenever $S>n$.

Stronger statements are available under additional assumptions. For
example, Cheng--Wei--Yamashiro \cite{ref4} proved
$S>1.8252n-0.712898$ for complete minimal hypersurfaces with constant
$S>n$, constant $f_3$, and $n\geq5$. Results under a Willmore condition and
under constant Gauss--Kronecker curvature in dimension four appear in
\cite{ref7,ref5}, respectively. Tan--Tang--Xie--Yan \cite{ref13} study local
finiteness of the values of $S$ for closed embedded minimal hypersurfaces
with constant $S$ and constant $f_3$. These hypotheses and conclusions
differ from the estimate below.

\begin{theorem}[Second pinching]\label{thm:1.1}
Let $F:M^n\to\mathbb{S}^{n+1}(1)$, $n\geq4$, be a closed connected
minimally immersed hypersurface with constant $S$. If $S>n$, then
\begin{equation}
 S>\frac{85}{59}n.
 \label{eq:1.4}
\end{equation}
Consequently, if $S\leq85n/59$, then $S=0$ or $S=n$. Up to an ambient
isometry, the image of $F$ is an equatorial sphere or one of the Clifford
hypersurfaces \eqref{eq:1.2}, respectively.
\end{theorem}

Theorem~\ref{thm:1.1} imposes no condition on $f_3$. Its improvement over
the Suh--Yang threshold is
\[
 \frac{85}{59}-\frac{10}{7}=\frac{5}{413}.
\]
The argument below treats all $0<(S-n)/S\leq26/85$ directly, so the
Suh--Yang estimate is used for comparison and is not an input to the proof.

For $S>0$, define the quadratic defect tensor and its squared norm by
\begin{equation}
 \Phi=A^2-\frac{f_3}{S}A-\frac{S}{n}I,\qquad
 P=|\Phi|^2=f_4-\frac{f_3^2}{S}-\frac{S^2}{n}.
 \label{eq:1.5}
\end{equation}
Thus $P$ measures the failure of the principal curvatures to satisfy a
common quadratic equation. In particular, $P=0$ at a point if and only if
there are exactly two distinct principal curvatures there. For a function
$v$ on $M$, write
\begin{equation}
 \langle v\rangle_M=\frac{1}{\Vol(M)}\int_M v\,d\mu.
 \label{eq:1.6}
\end{equation}

\begin{theorem}[Average defect gap]\label{thm:1.2}
Let $F:M^n\to\mathbb{S}^{n+1}(1)$, $n\geq2$, be a closed connected
minimally immersed hypersurface with constant $S>n$. Define
\begin{equation}
 c_n=\sqrt{\frac{n-1}{n}},\quad
 L_n=\frac{4}{3}+2c_n,\quad
 Q_n=\frac{11n+4}{n+2},\quad C_n=\frac{4}{3n},
 \label{eq:1.7}
\end{equation}
and
\begin{equation}
 \gamma_n=\frac{\sqrt{L_n^2+4Q_nC_n}-L_n}{2Q_n}.
 \label{eq:1.8}
\end{equation}
Then
\begin{equation}
 \left\langle\frac{P}{S^2}\right\rangle_M\geq\gamma_n^2>0.
 \label{eq:1.9}
\end{equation}
Moreover, $n^2\gamma_n^2\to4/25$ as $n\to\infty$.
\end{theorem}

The defect in \eqref{eq:1.5} has antecedents in the moment methods of
Yang--Cheng \cite{ref15} and Cheng--Wei--Yamashiro \cite{ref4}. It is
related to the pairwise invariant used by Ding--Xin \cite{ref6} through
\begin{equation}
 G:=\sum_{i,j}(\lambda_i-\lambda_j)^2(1+\lambda_i\lambda_j)^2
 =2SP+\frac{2S}{n}(S-n)^2.
 \label{eq:1.10}
\end{equation}
The mixed moment $Sf_4/4-cf_3^2/6$ and a refined cubic projection
estimate appear in the self-shrinker work of Cheng--Wei \cite{ref3};
related symmetrization arguments also occur in \cite{ref2}. We use these
established devices with the spherical curvature identities, retain the
terms involving the variation of $f_3$, and estimate the repeated-index
components under both trace constraints. The resulting inequality closes
with $c=89/67$ and the moment bound $2P/S^2+f_3^2/S^3\leq1-2/n$.

Section~\ref{sec:2} fixes conventions and records the identities.
Section~\ref{sec:3} proves the tensor estimates. Section~\ref{sec:4} proves
Theorem~\ref{thm:1.1}, and Section~\ref{sec:5} proves
Theorem~\ref{thm:1.2} together with a refinement involving a lower bound
for $|f_3|/S^{3/2}$.

\section{Preliminaries and moment identities}
\label{sec:2}

\subsection{Conventions and curvature formulas}
We use the Laplacian $\Delta=\tr\nabla^2=\operatorname{div}\nabla$.
All tensor inner products use full ordered sums of components. Roman
indices range from $1$ to $n$, and every sum in a component formula is
displayed explicitly. At a fixed point, choose an orthonormal frame for
which $h_{ij}=\lambda_i\delta_{ij}$. Covariant differentiations are
performed before making this choice; no smooth choice of eigenvectors
is required.

If a global unit normal is unavailable, calculations are made on the
unit normal double cover. Reversing the normal changes the signs of $A$,
$f_3$, and $\mathcal{C}$ defined below, whereas $\Phi$, $P$, $f_3^2$, and
every scalar integrand used in the proofs are unchanged. The integral
identities therefore descend to $M$, and normalized averages are unchanged
by this cover.

Write $h_{ijk}=(\nabla_kh)_{ij}$ and
$h_{ijkl}=(\nabla_l\nabla_kh)_{ij}$. The Gauss, Codazzi, and Ricci
equations, with our convention, are
\begin{align}
 R_{ijkl}&=\delta_{ik}\delta_{jl}-\delta_{il}\delta_{jk}
              +h_{ik}h_{jl}-h_{il}h_{jk},\label{eq:2.1}\\
 h_{ijk}&=h_{ikj},\label{eq:2.2}\\
 h_{ijkl}-h_{ijlk}&=\sum_mh_{im}R_{mjkl}
                     +\sum_mh_{mj}R_{mikl}.\label{eq:2.3}
\end{align}
In particular, $h_{ijk}$ is fully symmetric. Simons' formula \cite{ref11}
reads
\begin{equation}
 \Delta h=(n-S)h,\qquad
 \frac12\Delta S=|\nabla h|^2+S(n-S).
 \label{eq:2.4}
\end{equation}
Whenever $S>n$ is constant, put
\begin{equation}
 t=\frac{S-n}{S}\in(0,1),\qquad
 q=|\nabla h|^2=S(S-n)=tS^2.
 \label{eq:2.5}
\end{equation}
Then $q$ is constant. Minimality and constancy of $S$ give, for each $k$,
\begin{equation}
 \sum_i\lambda_i=0,\qquad \sum_i h_{iik}=0,\qquad
 \sum_i\lambda_i h_{iik}=0.
 \label{eq:2.6}
\end{equation}

Define the contractions
\begin{align}
 \mathcal{A}&=\sum_{i,j,k}\lambda_i^2h_{ijk}^2,\qquad
 \mathcal{B}=\sum_{i,j,k}\lambda_i\lambda_jh_{ijk}^2,
 \label{eq:2.7}\\
 \mathcal{C}&=\sum_{i,j,k}\lambda_i h_{ijk}^2,\qquad
 D_l=\sum_i\lambda_i^2h_{iil},\qquad
 J=\sum_l\lambda_l^2D_l^2.
 \label{eq:2.8}
\end{align}
The scalar $\mathcal{A}$ is distinct from the shape operator $A$.

\begin{lemma}[Differential moment identities]\label{lem:2.1}
For a minimal hypersurface of the unit sphere,
\begin{align}
 \frac13\Delta f_3&=(n-S)f_3+2\mathcal{C},
 \label{eq:2.9}\\
 \frac14\Delta f_4&=(n-S)f_4+2\mathcal{A}+\mathcal{B},
 \label{eq:2.10}\\
 (f_3)_{;l}&=3D_l,\qquad |\nabla f_3|^2=9\sum_lD_l^2.
 \label{eq:2.11}
\end{align}
\end{lemma}
\begin{proof}
For $A_{;k}=\nabla_kA$, the product rule and cyclicity of the trace give
\[
 (f_r)_{;k}=r\tr(A^{r-1}A_{;k}).
\]
Differentiating once more and using \eqref{eq:2.4}, for $r\geq2$ we obtain
\[
 \Delta f_r=r(n-S)f_r+
 r\sum_k\sum_{\ell=0}^{r-2}\tr(A^\ell A_{;k}A^{r-2-\ell}A_{;k}).
\]
For $r=3$, the two terms in the inner sum each contribute $\mathcal{C}$
after summing over $k$. For $r=4$, the two outer terms contribute
$\mathcal{A}$ each, and the middle term contributes $\mathcal{B}$.
These observations give \eqref{eq:2.9} and \eqref{eq:2.10}. The first
derivative formula with $r=3$ gives \eqref{eq:2.11}.
\end{proof}

The Peng--Terng formula \cite{ref9,ref10}, also recorded in
\cite[\S2, (2.2)]{ref6}, is
\begin{equation}
 \frac12\Delta|\nabla h|^2=|\nabla^2h|^2+(2n+3-S)|\nabla h|^2
 +3(2\mathcal{B}-\mathcal{A})-\frac32|\nabla S|^2.
 \label{eq:2.12}
\end{equation}
For constant $S>n$, both $\nabla S$ and $\Delta q$ vanish. Hence
\begin{equation}
 |\nabla^2h|^2=S(S-n)(S-2n-3)+3(\mathcal{A}-2\mathcal{B}).
 \label{eq:2.13}
\end{equation}

\subsection{Algebra of the defect}
Since $\tr A=0$ and $\tr A^2=S$, the tensor in \eqref{eq:1.5} satisfies
\begin{equation}
 \langle\Phi,I\rangle=\langle\Phi,A\rangle=0,\qquad
 \langle\Phi,A^2\rangle=P.
 \label{eq:2.14}
\end{equation}
At a principal frame,
\begin{equation}
 \varphi_i=\lambda_i^2-\frac{f_3}{S}\lambda_i-\frac{S}{n},\qquad
 P=\sum_i\varphi_i^2.
 \label{eq:2.15}
\end{equation}
This proves the nonnegativity asserted in \eqref{eq:1.5}. If $P=0$, all
$\lambda_i$ are roots of $s^2-(f_3/S)s-S/n$; the roots are distinct and
have opposite signs. Minimality and $S>0$ require both roots to occur.
Conversely, if there are two distinct values, their common monic
quadratic has constant coefficient $-S/n$ by taking its trace and linear
coefficient $-f_3/S$ by pairing with $A$. Thus $P=0$.

\begin{lemma}[Pairwise identity]\label{lem:2.2}
For a minimal hypersurface with $S>0$,
\begin{equation}
 \begin{aligned}
  G&=2(Sf_4-f_3^2-2S^2+nS)\\
   &=2SP+\frac{2S}{n}(S-n)^2.
 \end{aligned}
 \label{eq:2.16}
\end{equation}
\end{lemma}
\begin{proof}
Expand $(\lambda_i-\lambda_j)^2(1+\lambda_i\lambda_j)^2$ and sum over
ordered pairs. The terms of degrees two, four, and six sum to $2nS$,
$-4S^2$, and $2Sf_4-2f_3^2$, respectively, because $\sum_i\lambda_i=0$.
Substitute $f_4=P+f_3^2/S+S^2/n$ for the second equality.
\end{proof}

Normalize the moments by
\begin{equation}
 \mu_i=\frac{\lambda_i}{\sqrt S},\quad
 m_r=\sum_i\mu_i^r,\quad p=\frac{P}{S^2},\qquad
 z=\frac{f_3^2}{S^3}=m_3^2.
 \label{eq:2.17}
\end{equation}
Then $m_1=0$, $m_2=1$, and $p=m_4-m_3^2-1/n$.

\begin{lemma}[Moment inequality]\label{lem:2.3}
At each point with $S>0$,
\begin{equation}
 p\geq0,\qquad z\geq0,\qquad 2p+z\leq1-\frac2n.
 \label{eq:2.18}
\end{equation}
\end{lemma}
\begin{proof}
The first two assertions are immediate. For the third, the identities
\[
 \sum_{i<j}\mu_i^2\mu_j^2(\mu_i+\mu_j)^2=m_2m_4+m_3^2-2m_6,
\]
and
\[
 6\sum_{i<j<k}\mu_i^2\mu_j^2\mu_k^2=m_2^3-3m_2m_4+2m_6
\]
give, on using $m_2=1$,
\begin{equation}
 1-2m_4+m_3^2=\sum_{i<j}\mu_i^2\mu_j^2(\mu_i+\mu_j)^2
 +6\sum_{i<j<k}\mu_i^2\mu_j^2\mu_k^2\geq0.
 \label{eq:2.19}
\end{equation}
Since $2p+z=2m_4-m_3^2-2/n$, this is \eqref{eq:2.18}.
\end{proof}

\section{Tensor estimates}
\label{sec:tensor-estimates}\label{sec:3}

Throughout this section, $S>n$ is constant and $q,t$ are given by \eqref{eq:2.5}.

\subsection{A symmetrized Hessian estimate}
Define the full symmetrization of the Hessian by
\begin{equation}
 U_{ijkl}=\frac14(h_{ijkl}+h_{jkli}+h_{klij}+h_{lijk}).
 \label{eq:3.1}
\end{equation}
The symmetry of $h_{ijk}$ implies that $U$ is fully symmetric.

\begin{lemma}[Hessian decomposition]\label{lem:hessian-decomposition}\label{lem:3.1}
The tensors $\nabla^2 h$ and $U$ satisfy
\begin{equation}
 |\nabla^2 h|^2=|U|^2+\frac34G.
 \label{eq:3.2}
\end{equation}
\end{lemma}
\begin{proof}
At a principal frame, \eqref{eq:2.3} becomes
\begin{equation}
 h_{ijkl}-h_{ijlk}=(\lambda_i-\lambda_j)(1+\lambda_i\lambda_j)
 (\delta_{ik}\delta_{jl}-\delta_{il}\delta_{jk}).
 \label{eq:3.3}
\end{equation}
Since $\nabla^2 h$ is symmetric in its first three indices, \eqref{eq:3.1} is its orthogonal projection onto the fully symmetric tensors. The only possible nonzero components of $\nabla^2 h-U$ have two occurrences of each of two distinct indices. For $i\ne k$, put
\[
 r_{ik}=h_{iikk}-h_{kkii}=(\lambda_i-\lambda_k)(1+\lambda_i\lambda_k).
\]
Among the six ordered arrangements of $i,i,k,k$, three Hessian components equal $h_{iikk}$ and three equal $h_{kkii}$. Their symmetrized value is their arithmetic mean. Their contribution to $|\nabla^2 h-U|^2$ is therefore $6r_{ik}^2/4$. Summing over $i<k$ gives $3G/4$, proving \eqref{eq:3.2}. This is the four-cycle decomposition used in \cite[\S3]{ref4}.
\end{proof}

\begin{proposition}[Refined Hessian estimate]\label{prop:refined-hessian}\label{prop:3.2}
One has
\begin{equation}
 |\nabla^2 h|^2\ge 2SP-2\mathcal{A}
 +\frac{2t(1+t)S^3}{n}+\frac{4J}{tS^2}
 +\frac{2f_3}{S}\mathcal{C}+\frac{\mathcal{C}^2}{S^2}.
 \label{eq:3.4}
\end{equation}
\end{proposition}
\begin{proof}
Put $a=f_3/S$ and identify $\Phi$ with its symmetric bilinear form. Differentiating $S$ twice, the trace of $h$ twice, and $q$ once gives
\begin{align}
 \sum_i\lambda_i h_{iikk}&=-\sum_{i,j}h_{ijk}^2,
 &\sum_i h_{iikk}&=0,\label{eq:3.5}\\
 \sum_{i,j,k}h_{ijk}h_{ijkl}&=0.&&\label{eq:3.6}
\end{align}
Consequently,
\begin{align}
 \sum_{i,k}\lambda_i\lambda_k h_{iikk}&=-\mathcal{C},
 &\sum_{i,k}\lambda_i\lambda_k^2 h_{iikk}&=-\mathcal{A},
 \label{eq:3.7}\\
 \sum_{i,k}\lambda_i h_{iikk}&=-q,
 &\sum_{i,k}\lambda_k h_{iikk}&=0.
 \label{eq:3.8}
\end{align}
Commuting the two pairs of indices gives the remaining contraction:
\begin{equation}
 \sum_{i,k}\lambda_i^2\lambda_k h_{iikk}
 =-\mathcal{A}+SP+\frac{tS^3}{n}.
 \label{eq:3.9}
\end{equation}
Indeed, the difference between the left side of \eqref{eq:3.9} and the second sum in \eqref{eq:3.7} is
\[
 \sum_{i,k}\lambda_i^2\lambda_k(\lambda_i-\lambda_k)(1+\lambda_i\lambda_k)
 =Sf_4-f_3^2-S^2=SP+\frac{tS^3}{n}.
\]
It follows that
\begin{equation}
 K:=\langle U,h\otimes\Phi\rangle
 =-\mathcal{A}+\frac12SP+a\mathcal{C}+\frac{tS^3}{n}.
 \label{eq:3.10}
\end{equation}
For example, $U_{iikk}=(h_{iikk}+h_{kkii})/2$; substituting $\varphi_k=\lambda_k^2-a\lambda_k-S/n$ into $\sum_{i,k}U_{iikk}\lambda_i\varphi_k$ and using \eqref{eq:3.7}--\eqref{eq:3.9} gives \eqref{eq:3.10}.

In the space of fourth-order tensors, set
\[
\begin{aligned}
 X_{ijkl}&=h_{ij}\Phi_{kl}+\Phi_{ij}h_{kl},
 &Z_{ijkl}&=h_{ij}h_{kl},\\
 W_{ijkl}&=h_{ij}\delta_{kl}+\delta_{ij}h_{kl},
\end{aligned}
\]
and, for a covector $\xi=(\xi_l)$, set
\[
 Y(\xi)_{ijkl}=\xi_i h_{jkl}+\xi_j h_{ikl}+\xi_k h_{ijl}+\xi_l h_{ijk}.
\]
Direct contractions using \eqref{eq:2.6} and \eqref{eq:2.14} give
\begin{align}
 |X|^2&=2SP,& |Z|^2&=S^2,& |W|^2&=2nS,
 \label{eq:3.11}\\
 \langle U,X\rangle&=2K,&\langle U,Z\rangle&=-\mathcal{C},&
 \langle U,W\rangle&=-q.
 \label{eq:3.12}
\end{align}
Among $X,Y(\xi),Z,W$, all distinct pairwise inner products vanish except
\begin{equation}
 \langle X,Y(\xi)\rangle=4\sum_l\xi_l\lambda_lD_l.
 \label{eq:3.13}
\end{equation}
Here $D_l=\sum_i\varphi_i h_{iil}$ by \eqref{eq:2.6}.

We verify the contraction with $U$ that is sensitive to the variation of $f_3$. From \eqref{eq:3.3} and \eqref{eq:2.6},
\[
 \sum_{i,j,k}h_{ijk}(h_{ijkl}-h_{ijlk})
 =2\sum_i(\lambda_i-\lambda_l)(1+\lambda_i\lambda_l)h_{iil}
 =2\lambda_lD_l.
\]
The first summand on the left vanishes by \eqref{eq:3.6}. Full symmetry in the first three indices now yields
\[
 \sum_{i,j,k}U_{ijkl}h_{ijk}
 =\frac34\sum_{i,j,k}h_{ijlk}h_{ijk}=-\frac32\lambda_lD_l.
\]
Thus
\begin{equation}
 \langle U,Y(\xi)\rangle=-6\sum_l\xi_l\lambda_lD_l.
 \label{eq:3.14}
\end{equation}
If $\mathsf{R}_{lm}=\sum_{i,j}h_{lij}h_{mij}$, then $\mathsf{R}$ is positive semidefinite with trace $q$. Expanding the four terms in $Y$ therefore gives
\begin{equation}
 |Y(\xi)|^2=4q|\xi|^2+12\xi^{\mathsf{T}}\mathsf{R}\xi\le16q|\xi|^2.
 \label{eq:3.15}
\end{equation}
Choose the fixed coefficients
\begin{equation}
 \xi_l=\frac{\lambda_lD_l}{2q},\qquad
 \gamma=\frac{\mathcal{C}}{S^2},\qquad
 \eta=\frac{q}{2nS}.
 \label{eq:3.16}
\end{equation}
Expand $|U-X/2+Y(\xi)+\gamma Z+\eta W|^2$ using \eqref{eq:3.11}--\eqref{eq:3.15}. Since $\sum_l\xi_l\lambda_lD_l=J/(2q)$ and $16q|\xi|^2=4J/q$, one obtains
\begin{equation}
\begin{aligned}
 0&\le |U-X/2+Y(\xi)+\gamma Z+\eta W|^2\\
  &\le |U|^2+\frac12SP-2K-\frac{4J}{q}-\frac{\mathcal{C}^2}{S^2}-\frac{q^2}{2nS}.
\end{aligned}
 \label{eq:3.17}
\end{equation}
Hence \eqref{eq:3.10} gives
\begin{equation}
 |U|^2\ge\frac12SP-2\mathcal{A}+2a\mathcal{C}
 +\frac{2tS^3}{n}+\frac{4J}{q}+\frac{\mathcal{C}^2}{S^2}+\frac{q^2}{2nS}.
 \label{eq:3.18}
\end{equation}
Finally, add $3G/4$ by Lemma~3.1 and use Lemma~2.2. The two additional scalar terms are
\[
 \frac{q^2}{2nS}+\frac{3S}{2n}(S-n)^2=\frac{2t^2S^3}{n}.
\]
This proves \eqref{eq:3.4}. Only $q>0$ and $S>0$ were used in the denominators, so the proof also applies at points where $P=0$.
\end{proof}

\subsection{A refined cubic projection}
The following algebraic estimate is the diagonal-correction inequality of Cheng--Wei \cite[Lemma~3.5]{ref3}. We give a projection proof to fix the normalization.

\begin{lemma}[Cubic projection]\label{lem:cubic-projection}\label{lem:3.3}
One has
\begin{equation}
 \mathcal{A}+2\mathcal{B}\ge\frac{3\mathcal{C}^2}{q}
 +4\sum_l\frac{D_l^2}{S+2\lambda_l^2}.
 \label{eq:3.19}
\end{equation}
\end{lemma}
\begin{proof}
In the space of fully symmetric cubic tensors, let
\[
 T_{ijk}=h_{ijk},\qquad V_{ijk}=(\lambda_i+\lambda_j+\lambda_k)h_{ijk},
\]
and define, for $1\le l\le n$,
\[
 E_{ijk}^{(l)}=\frac13\bigl(\lambda_i\delta_{ij}\delta_{kl}
 +\lambda_j\delta_{jk}\delta_{il}+\lambda_k\delta_{ki}\delta_{jl}\bigr).
\]
Equation \eqref{eq:2.6} implies $\langle T,E^{(l)}\rangle=0$. Direct summation gives
\[
\begin{aligned}
 |V|^2&=3(\mathcal{A}+2\mathcal{B}),
 &\langle V,T\rangle&=3\mathcal{C},\\
 \langle V,E^{(l)}\rangle&=2D_l,
 &\langle E^{(l)},E^{(m)}\rangle&=\delta_{lm}\frac{S+2\lambda_l^2}{3}.
\end{aligned}
\]
Apply Bessel's inequality to the mutually orthogonal nonzero tensors $T,E^{(1)},\ldots,E^{(n)}$. Since $|T|^2=q$, it gives
\[
 3(\mathcal{A}+2\mathcal{B})\ge\frac{9\mathcal{C}^2}{q}
 +12\sum_l\frac{D_l^2}{S+2\lambda_l^2}.
\]
Division by $3$ proves the assertion.
\end{proof}

\subsection{The repeated-index estimate}
The two constraints in \eqref{eq:2.6} give a useful bound on the contraction $D_k$ when the multiplicities of the repeated indices are retained.

\begin{lemma}[Weighted repeated-index estimate]\label{lem:weighted-repeated-index}\label{lem:3.4}
Set
\begin{equation}
 x_k=\frac{\lambda_k^2}{S},\quad
 d_k=\frac{D_k}{S^2},\quad
 Q_k=h_{kkk}^2+3\sum_{i\ne k}h_{iik}^2,\quad
 \theta_k=\frac{Q_k}{S^2}.
 \label{eq:3.20}
\end{equation}
Then
\begin{equation}
 d_k^2\le\frac{np}{n+2+2nx_k}\theta_k,\qquad
 \sum_k\theta_k\le t.
 \label{eq:3.21}
\end{equation}
\end{lemma}
\begin{proof}
Fix $k$ and let $E=\{\mathbf{1},\boldsymbol{\lambda}\}^{\perp}\subset\mathbb{R}^n$, where $\boldsymbol{\lambda}=(\lambda_1,\ldots,\lambda_n)$. The two spanning vectors of $E^{\perp}$ are orthogonal and have squared norms $n$ and $S$. Both
\[
 v=(h_{11k},\ldots,h_{nnk})\quad\text{and}\quad
 \varphi=(\varphi_1,\ldots,\varphi_n)
\]
belong to $E$, and $D_k=\langle\varphi,v\rangle$. Let $u$ be the orthogonal projection of the $k$th coordinate vector onto $E$. Then
\begin{equation}
 |u|^2=1-\frac1n-x_k,\qquad
 Q_k=3|v|^2-2\langle u,v\rangle^2.
 \label{eq:3.22}
\end{equation}
The positive definite operator $W_k=3I_E-2u\otimes u$ represents this quadratic form on $E$. Weighted Cauchy--Schwarz gives
\begin{equation}
 D_k^2\le\langle\varphi,W_k^{-1}\varphi\rangle\,Q_k.
 \label{eq:3.23}
\end{equation}
The inverse of a rank-one perturbation satisfies
\[
 W_k^{-1}=\frac13I_E+\frac{2}{3(3-2|u|^2)}\,u\otimes u.
\]
Since $\langle u,\varphi\rangle=\varphi_k$ and $|\varphi|^2=P$, this yields
\begin{equation}
\begin{aligned}
 \langle\varphi,W_k^{-1}\varphi\rangle
 &=\frac{P}{3}+\frac{2n\varphi_k^2}{3(n+2+2nx_k)}\\
 &\le\frac{P}{3}+\frac{2n|u|^2P}{3(n+2+2nx_k)}
 =\frac{nP}{n+2+2nx_k}.
\end{aligned}
 \label{eq:3.24}
\end{equation}
The last equality uses $n+2+2nx_k+2n|u|^2=3n$. Equations \eqref{eq:3.23} and \eqref{eq:3.24}, divided by $S^4$, prove the first assertion in \eqref{eq:3.21}.

The quantities $Q_k$ count disjoint ordered components of the symmetric tensor $h_{ijk}$ with at least two equal indices. Therefore
\[
 \sum_k Q_k=\sum_k h_{kkk}^2+3\sum_{i\ne k}h_{iik}^2
 \le\sum_{i,j,k}h_{ijk}^2=q=tS^2,
\]
which proves the second assertion.
\end{proof}

\section{The integral inequality and second pinching}
\label{sec:4}

\subsection{The mixed moment}
Besides \eqref{eq:2.17} and \eqref{eq:3.20}, use the normalized scalars
\begin{equation}
\alpha=\frac{\mathcal{A}}{S^3},\quad
\beta=\frac{\mathcal{B}}{S^3},\quad
e=\frac{f_3\mathcal{C}}{S^4},\quad
w=\frac{\mathcal{C}^2}{S^5},\quad
d=\sum_i d_i^2,\quad j=\sum_i x_i d_i^2.
\label{eq:4.1}
\end{equation}
Notice that $e^2=zw$. Combining \eqref{eq:3.4} with \eqref{eq:2.13} gives
\begin{equation}
H:=5\alpha-6\beta-2p-\frac{4}{t}j-2e-w+t(2t-1)-\frac{t(5-t)}{n}\geq0.
\label{eq:4.2}
\end{equation}
For the scalar terms, one uses
\[
\frac{q(S-2n-3)}{S^3}=t(2t-1)-\frac{3t(1-t)}{n}.
\]
For a real constant $c$, define
\begin{equation}
F_c=\frac{S}{4}f_4-\frac{c}{6}f_3^2,\qquad
L_c=\frac{\Delta F_c}{S^4}.
\label{eq:4.3}
\end{equation}
Lemma~\ref{lem:2.1} and constancy of $S$ give
\begin{equation}
L_c=2\alpha+\beta-t\left(p+(1-c)z+\frac1n\right)-2ce-3cd.
\label{eq:4.4}
\end{equation}
Indeed, $\Delta(f_3^2)=2f_3\Delta f_3+2|\nabla f_3|^2$; thus the last term in \eqref{eq:4.4} retains the full contribution from $\nabla f_3$. Since $M$ is closed,
\begin{equation}
\langle L_c\rangle_M=0.
\label{eq:4.5}
\end{equation}
Fix
\begin{equation}
\kappa=\frac{16}{3},\qquad c=\frac{89}{67},\qquad
R=2\kappa c-2=\frac{2446}{201}.
\label{eq:4.6}
\end{equation}
It follows that
\begin{equation}
\mathcal{I}:=\langle H-\kappa L_c\rangle_M=\langle H\rangle_M\geq0.
\label{eq:4.7}
\end{equation}
This nonnegativity is an integral statement. No pointwise sign is asserted for $H-\kappa L_c$.

For $0<t<1$, set
\begin{align}
M_t(x)&=16c-\frac{4x}{t}-\frac{68}{3(1+2x)},\qquad x\geq0,
\label{eq:4.8}\\
\mathscr{R}(t)&=\sup_{x\geq0}\frac{(M_t(x))_+}{1+2x},
\label{eq:4.9}\\
\mathfrak{a}(t)&=\kappa t-2+t\mathscr{R}(t),\qquad
\mathfrak{b}(t)=\frac{t(292945-70752t)}{40401(t+17)}.
\label{eq:4.10}
\end{align}

\begin{proposition}[Integral estimate]\label{prop:4.1}
With the notation above, one has
\begin{equation}
0\leq\mathcal{I}\leq\mathfrak{a}(t)\langle p\rangle_M+\mathfrak{b}(t)\langle z\rangle_M+t(2t-1)+\frac{t(t+1/3)}{n}.
\label{eq:4.11}
\end{equation}
\end{proposition}

\begin{proof}
Expanding \eqref{eq:4.2} and \eqref{eq:4.4} with \eqref{eq:4.6} gives
\begin{equation}
\begin{aligned}
H-\kappa L_c={}&-\frac{17}{3}(\alpha+2\beta)+(\kappa t-2)p+\kappa t(1-c)z+Re-w\\
&+\sum_i\left(16c-\frac{4x_i}{t}\right)d_i^2+t(2t-1)+\frac{t(t+1/3)}{n}.
\end{aligned}
\label{eq:4.12}
\end{equation}
After dividing \eqref{eq:3.19} by $S^3$, Lemma~\ref{lem:3.3} reads
\begin{equation}
\alpha+2\beta\geq\frac{3w}{t}+4\sum_i\frac{d_i^2}{1+2x_i}.
\label{eq:4.13}
\end{equation}
Also, $e^2=zw$ and $R>0$ imply
\begin{equation}
Re-\left(1+\frac{17}{t}\right)w
\leq R\sqrt{zw}-\frac{t+17}{t}w
\leq\frac{tR^2}{4(t+17)}z.
\label{eq:4.14}
\end{equation}
The last inequality follows by completing the square in $\sqrt{w}$ and also holds when $z=0$.

For the remaining diagonal terms, Lemma~\ref{lem:3.4} gives
\begin{equation}
\begin{aligned}
\sum_i M_t(x_i)d_i^2
&\leq\sum_i(M_t(x_i))_+\frac{np}{n+2+2nx_i}\theta_i\\
&\leq p\mathscr{R}(t)\sum_i\theta_i\leq tp\mathscr{R}(t),
\end{aligned}
\label{eq:4.15}
\end{equation}
because $n/(n+2+2nx)\leq1/(1+2x)$. Inserting \eqref{eq:4.13}--\eqref{eq:4.15} into \eqref{eq:4.12} and averaging proves \eqref{eq:4.11}, since exact simplification gives
\[
\kappa t(1-c)+\frac{tR^2}{4(t+17)}
=\frac{t(292945-70752t)}{40401(t+17)}.
\]
\end{proof}

\subsection{Exact coefficient estimates}
We now estimate the coefficients on the whole interval needed for the proof, including values of $t$ below the Suh--Yang threshold.

\begin{lemma}[Coefficient bounds]\label{lem:4.2}
For $0<t\leq26/85$, one has
\begin{equation}
\mathfrak{b}(t)>0,\qquad \mathfrak{a}(t)<2\mathfrak{b}(t),
\label{eq:4.16}
\end{equation}
and
\begin{align}
F_0(t)&:=\mathfrak{b}(t)+t(2t-1)<0,
\label{eq:4.17}\\
F_1(t)&:=t(t+1/3)-2\mathfrak{b}(t)<0.
\label{eq:4.18}
\end{align}
\end{lemma}

\begin{proof}
Positivity of $\mathfrak{b}$ is immediate from \eqref{eq:4.10}. To estimate $\mathscr{R}$, put $y=1+2x\geq1$. Before taking the positive part, the expression in \eqref{eq:4.9} becomes
\begin{equation}
\frac{16c+2/t}{y}-\frac2t-\frac{68}{3y^2}.
\label{eq:4.19}
\end{equation}
For $t\geq1/5$, its maximum occurs at
\begin{equation}
y_* =\frac{4556t}{3(712t+67)}>1.
\label{eq:4.20}
\end{equation}
Indeed, its derivative has the sign of $136/3-(16c+2/t)y$, and $y_*>1$ is equivalent to $t>201/2420$. The value at $y_*$ is positive for $t\geq1/5$, and therefore
\begin{equation}
\mathscr{R}(t)=\frac{1520832t^2-324280t+13467}{305252t^2},\qquad t\geq\frac15.
\label{eq:4.21}
\end{equation}
To check positivity, the numerator equals $236107/25$ at $t=1/5$, and its derivative is $3041664t-324280>0$ on this range.

For $0<t\leq1/5$, $M_t(x)\leq M_{1/5}(x)$ for every $x\geq0$. Hence $\mathscr{R}(t)\leq\mathscr{R}(1/5)=236107/305252$, and
\begin{equation}
\begin{aligned}
\mathfrak{a}(t)&\leq\left(\frac{16}{3}+\frac{236107}{305252}\right)t-2\\
&\leq-\frac{3565207}{4578780}<0<2\mathfrak{b}(t).
\end{aligned}
\label{eq:4.22}
\end{equation}
It remains to consider $1/5\leq t\leq26/85$. In this range \eqref{eq:4.21} yields
\begin{equation}
\mathfrak{a}(t)=\frac{9446528t^2-2804352t+40401}{915756t}.
\label{eq:4.23}
\end{equation}
Exact subtraction gives
\begin{equation}
\mathfrak{a}(t)-2\mathfrak{b}(t)=\frac{N(t)}{2747268t(t+17)},
\label{eq:4.24}
\end{equation}
where
\[
N(t)=37961856t^3+433519352t^2-142900749t+2060451.
\]
The polynomial $N$ is strictly convex for $t>0$, since
\[
N''(t)=16(14235696t+54189919)>0.
\]
Its endpoint values are
\begin{equation}
N(1/5)=-\frac{1109403734}{125}<0,\qquad
N(26/85)=-\frac{1291682299}{614125}<0.
\label{eq:4.25}
\end{equation}
A convex function lies below its endpoint chord, so $N<0$ throughout $[1/5,26/85]$. This proves \eqref{eq:4.16}.

Finally,
\begin{align}
F_0(t)&=\frac{t(80802t^2+1262481t-393872)}{40401(t+17)},
\label{eq:4.26}\\
F_1(t)&=\frac{t(40401t^2+841788t-356951)}{40401(t+17)}.
\label{eq:4.27}
\end{align}
Both quadratic numerator factors are strictly increasing for $t\geq0$. At $t=26/85$ their values are, respectively,
\begin{equation}
-\frac{1020038}{7225}<0,\qquad -\frac{691308419}{7225}<0.
\label{eq:4.28}
\end{equation}
They are therefore negative for $0<t\leq26/85$. This proves \eqref{eq:4.17} and \eqref{eq:4.18}.
\end{proof}

\begin{proof}[Proof of Theorem~\ref{thm:1.1}]
Suppose that $n<S\leq85n/59$. Then $0<t\leq26/85$. By Lemmas~\ref{lem:2.3} and~\ref{lem:4.2},
\[
\mathfrak{a}(t)\langle p\rangle_M+\mathfrak{b}(t)\langle z\rangle_M
\leq\mathfrak{b}(t)\langle2p+z\rangle_M
\leq\mathfrak{b}(t)\left(1-\frac2n\right).
\]
Proposition~\ref{prop:4.1} gives the contradiction
\[
0\leq\mathcal{I}\leq\mathfrak{b}(t)\left(1-\frac2n\right)+t(2t-1)+\frac{t(t+1/3)}{n}
=F_0(t)+\frac{F_1(t)}{n}<0.
\]
Thus $t>26/85$, equivalently $S>85n/59$. If $S\leq n$, \eqref{eq:2.4} and constancy of $S$ force $S=0$ or $S=n$. The classification recalled in Section~\ref{sec:1} completes the proof.
\end{proof}

\begin{remark}\label{rem:4.3}
Every constant used in the coefficient argument is rational, and the sign checks are proved on intervals, not inferred from numerical sampling. The value $85/59$ is not asserted to be optimal for this method. In particular, Theorem~\ref{thm:1.1} does not establish \eqref{eq:1.3}.
\end{remark}

\section{The average defect gap}\label{sec:5}

This section applies in every dimension $n\geq 2$. Its proof uses Simons' equation, the
Codazzi symmetry, and the algebra of $\Phi$; it does not use the parameters or integral
inequality of Section~4.

\subsection{Localization near two roots}
Put $a=f_3/S$ and define
\begin{equation}
\pi(s)=s^2-as-\frac{S}{n}=(s-\lambda_+)(s-\lambda_-),
\qquad \lambda_+>\lambda_-.
\label{eq:5.1}
\end{equation}
At a fixed point let $\delta=\lambda_+-\lambda_-$. Then
\begin{equation}
\delta^2=a^2+\frac{4S}{n},\qquad P=\sum_i\pi(\lambda_i)^2.
\label{eq:5.2}
\end{equation}

\begin{lemma}[Root localization]\label{lem:5.1}
If $\sqrt{P}<\delta^2/4$, set
\begin{equation}
\rho=\frac{\delta-\sqrt{\delta^2-4\sqrt{P}}}{2}.
\label{eq:5.3}
\end{equation}
Then $\operatorname{dist}(\lambda_i,\{\lambda_-,\lambda_+\})\leq\rho$ for every $i$.
For any three indices $i,j,k$, including repeated indices, at least one of the three
pair sums satisfies
\begin{equation}
|\lambda_r+\lambda_s-a|^2\geq\delta^2-4\sqrt{P},
\qquad (r,s)\in\{(i,j),(i,k),(j,k)\}.
\label{eq:5.4}
\end{equation}
\end{lemma}

\begin{proof}
Equation~\eqref{eq:5.2} gives $|\pi(\lambda_i)|\leq\sqrt{P}$. If a real number
$\lambda$ lies between the two roots and $s$ is its distance to the nearer root, then
$0\leq s\leq\delta/2$ and $|\pi(\lambda)|=s(\delta-s)$. This function is increasing
on $[0,\delta/2]$, so $s(\delta-s)\leq\sqrt{P}$ implies $s\leq\rho$. Outside the
root interval, $|\pi(\lambda)|=s(\delta+s)$, which gives an even smaller upper bound
for $s$.

Assign each of three principal curvatures to its nearer root. Two are assigned to the
same root, say $\lambda_\pm$. Since $\lambda_++\lambda_-=a$ and
$|2\lambda_\pm-a|=\delta$, the corresponding pair satisfies
\[
|\lambda_r+\lambda_s-a|\geq\delta-2\rho=\sqrt{\delta^2-4\sqrt{P}}.
\]
This argument also covers repeated indices.
\end{proof}

Define the weighted Codazzi energy
\begin{equation}
\mathcal{T}=\sum_{i,j,k}(\lambda_i+\lambda_j-a)^2h_{ijk}^2.
\label{eq:5.5}
\end{equation}

\begin{proposition}[Three-pair estimate]\label{prop:5.2}
For constant $S>0$,
\begin{equation}
\mathcal{T}\geq\frac{1}{3}\left(a^2+\frac{4S}{n}-4\sqrt{P}\right)_+|\nabla h|^2.
\label{eq:5.6}
\end{equation}
\end{proposition}

\begin{proof}
Symmetry of $h_{ijk}$ gives
\begin{equation}
\begin{split}
\mathcal{T}=\frac{1}{3}\sum_{i,j,k}\bigl[& (\lambda_i+\lambda_j-a)^2
 +(\lambda_i+\lambda_k-a)^2\\
 & +(\lambda_j+\lambda_k-a)^2\bigr]h_{ijk}^2.
\end{split}
\label{eq:5.7}
\end{equation}
If $\sqrt{P}<\delta^2/4$, Lemma~\ref{lem:5.1} bounds each bracket below by
$\delta^2-4\sqrt{P}$. If $\sqrt{P}\geq\delta^2/4$, the claimed lower bound is zero.
Equation~\eqref{eq:5.2} completes the proof.
\end{proof}

\subsection{The Bochner formula with variable cubic trace}

\begin{lemma}[Gradient estimate]\label{lem:5.3}
For constant $S>0$,
\begin{equation}
a_{;k}=\frac{3}{S}\langle\Phi,\nabla_kh\rangle,
\qquad |\nabla a|^2\leq\frac{9nP}{(n+2)S^2}|\nabla h|^2.
\label{eq:5.8}
\end{equation}
\end{lemma}

\begin{proof}
By \eqref{eq:2.11} and \eqref{eq:2.6},
\[
a_{;k}=\frac{3}{S}\sum_i\lambda_i^2h_{iik}
 =\frac{3}{S}\sum_i\varphi_i h_{iik}.
\]
This proves the first identity. To estimate its norm, put
\[
A_0=\sum_i h_{iii}^2,\qquad C_0=\sum_{i\ne k}h_{iik}^2.
\]
For each $k$, the trace relation gives $h_{kkk}=-\sum_{i\ne k}h_{iik}$.
Cauchy--Schwarz therefore gives $A_0\leq(n-1)C_0$. Since
$|\nabla h|^2\geq A_0+3C_0$, it follows that
\begin{equation}
\sum_{i,k}h_{iik}^2=A_0+C_0\leq\frac{n}{n+2}(A_0+3C_0)
 \leq\frac{n}{n+2}|\nabla h|^2.
\label{eq:5.9}
\end{equation}
Applying Cauchy--Schwarz to the displayed formula for $a_{;k}$ and summing over $k$
proves \eqref{eq:5.8}.
\end{proof}

\begin{proposition}[Defect Bochner identity]\label{prop:5.4}
For a minimal hypersurface with constant $S>0$,
\begin{equation}
\frac{1}{2}\Delta P=\mathcal{T}+2(n-S)P
 +2\sum_k\langle\Phi,(\nabla_kA)^2\rangle-S|\nabla a|^2.
\label{eq:5.10}
\end{equation}
Here $(\nabla_kA)^2$ denotes composition of the symmetric endomorphism $\nabla_kA$
with itself. No constancy assumption on $f_3$ is made.
\end{proposition}

\begin{proof}
Differentiate \eqref{eq:1.5}. At a principal frame,
\begin{equation}
\Phi_{ij;k}=(\lambda_i+\lambda_j-a)h_{ijk}-a_{;k}\lambda_i\delta_{ij}.
\label{eq:5.11}
\end{equation}
Using $\sum_i\lambda_i h_{iik}=0$ and $D_k=Sa_{;k}/3$ in the cross term gives
\begin{equation}
|\nabla\Phi|^2=\mathcal{T}-4\sum_k a_{;k}D_k+S|\nabla a|^2
 =\mathcal{T}-\frac{S}{3}|\nabla a|^2.
\label{eq:5.12}
\end{equation}
Next, the product rule and Simons' equation give
\[
\begin{aligned}
\Delta(A^2)&=2(n-S)A^2+2\sum_k(\nabla_kA)^2,\\
\Delta(aA)&=(\Delta a)A+a(n-S)A+2\sum_k a_{;k}\nabla_kA.
\end{aligned}
\]
After pairing with $\Phi$, the terms proportional to $A$ vanish by
\eqref{eq:2.14}. Lemma~\ref{lem:5.3} then gives
\[
\langle\Phi,\Delta\Phi\rangle=2(n-S)P
 +2\sum_k\langle\Phi,(\nabla_kA)^2\rangle-\frac{2S}{3}|\nabla a|^2.
\]
Adding \eqref{eq:5.12} and using
$\Delta|\Phi|^2/2=|\nabla\Phi|^2+\langle\Phi,\Delta\Phi\rangle$ proves
\eqref{eq:5.10}.
\end{proof}

\begin{remark}\label{rem:5.5}
There is a direct scalar verification of \eqref{eq:5.10}. Applying Lemma~2.1 to
$P=f_4-f_3^2/S-S^2/n$ gives
\begin{equation}
\frac{1}{2}\Delta P=(n-S)\left(2f_4-\frac{3f_3^2}{S}\right)
 +4\mathcal{A}+2\mathcal{B}-6a\mathcal{C}-\frac{9}{S}\sum_kD_k^2.
\label{eq:5.13}
\end{equation}
Expanding the right side of \eqref{eq:5.10} gives the same expression, since
\[
\mathcal{T}=2\mathcal{A}+2\mathcal{B}-4a\mathcal{C}+a^2q,
\qquad \sum_k\langle\Phi,(\nabla_kA)^2\rangle
 =\mathcal{A}-a\mathcal{C}-\frac{S}{n}q.
\]
\end{remark}

For a trace-free symmetric endomorphism with eigenvalues $\varphi_i$,
Cauchy--Schwarz applied to $\varphi_i=-\sum_{j\ne i}\varphi_j$ gives
$\varphi_i^2\leq(n-1)P/n$. Consequently,
\begin{equation}
\lambda_{\min}(\Phi)\geq-c_n\sqrt{P}.
\label{eq:5.14}
\end{equation}
Each $(\nabla_kA)^2$ is positive semidefinite. Therefore
\begin{equation}
\sum_k\langle\Phi,(\nabla_kA)^2\rangle\geq-c_n\sqrt{P}|\nabla h|^2.
\label{eq:5.15}
\end{equation}

\begin{theorem}[Pointwise defect inequality]\label{thm:5.6}
Suppose that $S>n$ is constant, and set
\begin{equation}
\chi=\frac{\sqrt{P}}{S},\qquad \tau=\frac{|f_3|}{S^{3/2}}.
\label{eq:5.16}
\end{equation}
With $c_n,Q_n$ as in \eqref{eq:1.7}, one has
\begin{equation}
\frac{1}{2}\Delta P\geq S^2(S-n)\left[
 \frac{1}{3}\left(\tau^2+\frac{4}{n}-4\chi\right)_+
 -2c_n\chi-Q_n\chi^2\right].
\label{eq:5.17}
\end{equation}
\end{theorem}

\begin{proof}
Apply Proposition~\ref{prop:5.2}, Lemma~\ref{lem:5.3}, and \eqref{eq:5.15} to
\eqref{eq:5.10}. Since $a^2=S\tau^2$, $P=S^2\chi^2$, and $q=S(S-n)$, the four
terms in \eqref{eq:5.10}, after division by $S^2(S-n)>0$, are bounded below by
\[
\frac{1}{3}\left(\tau^2+\frac{4}{n}-4\chi\right)_+,
\qquad -2\chi^2,\qquad -2c_n\chi,\qquad -\frac{9n}{n+2}\chi^2,
\]
respectively. The quadratic coefficients add to
\[
2+\frac{9n}{n+2}=\frac{11n+4}{n+2}=Q_n.
\]
No derivative of $\chi$ is used, so the inequality remains valid at zeros of $P$.
\end{proof}

\subsection{Integration and the quantitative gap}
For $\tau_0\geq0$, put
\begin{equation}
C_n(\tau_0)=\frac{1}{3}\left(\tau_0^2+\frac{4}{n}\right),
\qquad \gamma_n(\tau_0)=\frac{\sqrt{L_n^2+4Q_nC_n(\tau_0)}-L_n}{2Q_n}.
\label{eq:5.18}
\end{equation}
This is the unique positive solution of
\begin{equation}
Q_n\gamma^2+L_n\gamma=C_n(\tau_0).
\label{eq:5.19}
\end{equation}
In particular, $C_n(0)=C_n$ and $\gamma_n(0)=\gamma_n$.

\begin{theorem}[Gap with a cubic-moment lower bound]\label{thm:5.7}
Let $F:M^n\to\mathbb{S}^{n+1}(1)$, $n\geq2$, be closed, connected, and minimal,
with constant $S>n$. If
\begin{equation}
\inf_M\frac{|f_3|}{S^{3/2}}\geq\tau_0\geq0,
\label{eq:5.20}
\end{equation}
then
\begin{equation}
\left\langle\frac{P}{S^2}\right\rangle_M\geq\gamma_n(\tau_0)^2.
\label{eq:5.21}
\end{equation}
\end{theorem}

\begin{proof}
Since $u_+\geq u$ for every real $u$, Theorem~\ref{thm:5.6} and \eqref{eq:5.20}
imply
\begin{equation}
\frac{1}{2}\Delta P\geq S^2(S-n)\bigl(C_n(\tau_0)-L_n\chi-Q_n\chi^2\bigr).
\label{eq:5.22}
\end{equation}
Integrate over $M$ and use $\int_M\Delta P\,d\mu=0$. With
$y=\sqrt{\langle\chi^2\rangle_M}$, Cauchy--Schwarz gives
$\langle\chi\rangle_M\leq y$, hence
\[
0\geq C_n(\tau_0)-L_n\langle\chi\rangle_M-Q_n\langle\chi^2\rangle_M
 \geq C_n(\tau_0)-L_ny-Q_ny^2.
\]
The function $Q_ny^2+L_ny$ is strictly increasing for $y\geq0$. By
\eqref{eq:5.19}, $y\geq\gamma_n(\tau_0)$, which proves \eqref{eq:5.21}.
\end{proof}

\begin{remark}[A scalar comparison]\label{rem:5.8}
The same integration can be expressed as a pointwise comparison. For
$\gamma=\gamma_n(\tau_0)$ and $K=Q_n+L_n/(2\gamma)$, one has for every $x\geq0$
\begin{equation}
\frac{1}{3}\left(\tau_0^2+\frac{4}{n}-4x\right)_+-2c_nx-Q_nx^2
 \geq K(\gamma^2-x^2).
\label{eq:5.23}
\end{equation}
Indeed, the left side is at least $C_n(\tau_0)-L_nx-Q_nx^2$, and the difference
between this quadratic and the right side is exactly $L_n(x-\gamma)^2/(2\gamma)$.
Thus no separate argument at the transition of the positive part is needed.
\end{remark}

\begin{proof}[Proof of Theorem~\ref{thm:1.2}]
Take $\tau_0=0$ in Theorem~\ref{thm:5.7}. Positivity of $\gamma_n$ follows from
$C_n>0$, $L_n>0$, and $Q_n>0$. For the asymptotic statement, \eqref{eq:5.19} gives
\[
\gamma_n=\frac{4/(3n)}{L_n+Q_n\gamma_n}.
\]
In particular, $0<\gamma_n\leq4/(3nL_n)$, so $\gamma_n\to0$. Since
$L_n\to10/3$ and $Q_n\to11$, it follows that $n\gamma_n\to2/5$, and therefore
$n^2\gamma_n^2\to4/25$.
\end{proof}

\begin{corollary}[Average pairwise energy]\label{cor:5.9}
Under the hypotheses of Theorem~\ref{thm:5.7},
\begin{equation}
\langle G\rangle_M\geq2S^3\gamma_n(\tau_0)^2+\frac{2S}{n}(S-n)^2.
\label{eq:5.24}
\end{equation}
\end{corollary}

\begin{proof}
Average \eqref{eq:2.16}, use that $S$ is constant, and apply \eqref{eq:5.21}.
\end{proof}

For each fixed dimension, Theorem~\ref{thm:1.2} excludes a sequence of constant-$S$
minimal hypersurfaces with $S>n$ for which $\langle P/S^2\rangle_M$ tends to zero.
The dimension must be fixed: the lower bound satisfies $\gamma_n^2\sim4/(25n^2)$.

\end{document}